\documentclass[a4paper, british, 11 pt]{amsart}

\usepackage{amssymb}
\usepackage{babel}
\usepackage{enumitem}
\usepackage[colorlinks=true, citecolor=blue]{hyperref}
\usepackage[utf8]{inputenc}
\usepackage[T1]{fontenc}
\usepackage{newunicodechar}
\usepackage{mathtools}
\usepackage{varioref}
\usepackage[arrow,curve,matrix]{xy}

\usepackage{colortbl}
\usepackage{graphicx}
\usepackage{tikz}
\usepackage{tikz-cd}

\usepackage{mathrsfs}

\definecolor{linkred}{rgb}{0.7,0.2,0.2}
\definecolor{linkblue}{rgb}{0,0.2,0.6}

\numberwithin{figure}{section}

\usepackage[hyperpageref]{backref}

\setdescription{labelindent=\parindent, leftmargin=2\parindent}
\setitemize[1]{labelindent=\parindent, leftmargin=2\parindent}
\setenumerate[1]{labelindent=0cm, leftmargin=*, widest=iiii}

\DeclareFontFamily{OMS}{rsfs}{\skewchar\font'60}
\DeclareFontShape{OMS}{rsfs}{m}{n}{<-5>rsfs5 <5-7>rsfs7 <7->rsfs10 }{}
\DeclareSymbolFont{rsfs}{OMS}{rsfs}{m}{n}
\DeclareSymbolFontAlphabet{\scr}{rsfs}
\DeclareSymbolFontAlphabet{\scr}{rsfs}

\DeclareFontFamily{U}{mathx}{\hyphenchar\font45}
\DeclareFontShape{U}{mathx}{m}{n}{
      <5> <6> <7> <8> <9> <10>
      <10.95> <12> <14.4> <17.28> <20.74> <24.88>
      mathx10
      }{}
\DeclareSymbolFont{mathx}{U}{mathx}{m}{n}
\DeclareFontSubstitution{U}{mathx}{m}{n}
\DeclareMathAccent{\wcheck}{0}{mathx}{"71}

\DeclareMathOperator{\Aut}{Aut}
\DeclareMathOperator{\Ann}{Ann}

\DeclareMathOperator{\rank}{rank}

\DeclareMathOperator{\red}{red}

\DeclareMathOperator{\Spec}{Spec}
\DeclareMathOperator{\Sym}{Sym}

\DeclareMathOperator{\Bs}{Bs}

\DeclareMathOperator{\Div}{Div}

\newcommand{\sA}{\scr{A}}

\newcommand{\sF}{\scr{F}}

\newcommand{\sM}{\scr{M}}

\newcommand{\sO}{\scr{O}}

\newcommand{\sT}{\scr{T}}

\newcommand{\bB}{\mathbb{B}}

\newcommand{\bN}{\mathbb{N}}

\newcommand{\bQ}{\mathbb{Q}}
\newcommand{\bR}{\mathbb{R}}

\theoremstyle{plain}
\newtheorem{thm}{Theorem}[section]

\newtheorem{cor}[thm]{Corollary}
\newtheorem{defn}[thm]{Definition}

\newtheorem{lem}[thm]{Lemma}

\newtheorem{prop}[thm]{Proposition}

\theoremstyle{remark}

\newtheorem{c-n-d}[thm]{Claim and Definition}

\newtheorem{construction}[thm]{Construction}

\newtheorem{convention}[thm]{Convention}

\newtheorem{notation}[thm]{Notation}

\newtheorem{rem}[thm]{Remark}

\newtheorem*{rem-nonumber}{Remark}
\newtheorem{setting}[thm]{Setting}

\numberwithin{equation}{thm}

\setlist[enumerate]{label=(\thethm.\arabic*), before={\setcounter{enumi}{\value{equation}}}, after={\setcounter{equation}{\value{enumi}}}}

\newcommand{\lra}{\longrightarrow}

\newcommand{\factor}[2]{\left. \raise 2pt\hbox{$#1$} \right/\hskip -2pt\raise -2pt\hbox{$#2$}}

\usepackage{libertine}
\usepackage[libertine]{newtxmath}
\usepackage{microtype}

\title{The stable and augmented base locus under finite morphisms}

\author{Tanuj Gomez}
\address{Tanuj Gomez, Mathematisches Institut, Albert-Ludwigs-Universität Freiburg, Ernst-Zermelo-Straße 1, 79104 Freiburg im Breisgau, Germany}
\email{\href{mailto:tanuj.gomez@math.uni-freiburg.de}{tanuj.gomez@hotmail.com}}
\thanks{Supported by the DFG-Graduiertenkolleg GK1821 ``Cohomological Methods in Geometry'' at the University of Freiburg.}
\date{\today}


\makeatletter
\@namedef{subjclassname@2020}{2020 Mathematics Subject Classification}
\makeatother

\makeatletter
\hypersetup{
  pdfauthor={\authors},
  pdftitle={\@title},
  pdfsubject={\@subjclass},
  pdfkeywords={\@keywords},
  pdfstartview={Fit},
  pdfpagelayout={TwoColumnRight},
  pdfpagemode={UseOutlines},
  bookmarks,
  colorlinks,
  linkcolor=linkblue,
  citecolor=linkred,
  urlcolor=linkred}
\makeatother

\begin{document}

\begin{abstract}
  We study the pullback of the stable and augmented base locus under a finite
  surjective morphism between normal varieties over a perfect field.
\end{abstract}

\maketitle

\section{introduction}

The \emph{stable} and \emph{augmented} base locus, denoted
$\bB(\cdot)$ and $\bB_+(\cdot)$,
are fundamental invariants in Birational Geometry and in Complex Geometry, particularly when
positivity and asymptotic behaviour of divisors, line bundles, or vector bundles
are discussed \cite{Naka00, BDPP12, ELMNP6}.

In this article, we study the behaviour of the stable and augmented
base locus under the pullback via a finite surjective morphism.
Previous results in this direction, \cite{dicerlom20} work in the setting of complex manifolds.
They, however, rely on powerful analytic methods such as currents and Lelong numbers \cite{Bouck04}.
If one is willing to impose strong positivity assumptions on the divisors, using \cite{Dicerbo19} and \cite{bir13},
these results can be extended to projective schemes. Our main result, Theorem \ref{mainthm2},
generalises these results to normal varieties over a perfect field.

\begin{thm}\label{mainthm2}
  Let $X$ and $Y$ be normal projective varieties over a perfect field $k$.
	Let $f: X \lra Y$ be a finite surjective morphism over $k$. Then the following hold:
	\begin{enumerate}
		\item For every $\bQ$-divisor $D$ on $Y$, we have an equality of reduced schemes
		$\bB(f^*D) = f^{-1}\bB(D)$.\label{mainthm2:1}
		\item For every $\bR$-divisor $\Delta$ on $Y$, we have an equality of reduced schemes
		$\bB_{+}(f^*\Delta) = f^{-1}\bB_{+}(\Delta)$.\label{mainthm2:2}
	\end{enumerate}
\end{thm}

In Section \ref{prelim}, we collect the definitions of $\bQ$-divisors, $\bR$-divisors,
the stable and augmented base locus.
In Section \ref{symsec}, we prove Theorem \ref{mainthm2} in the separable case by
passing to the Galois closure. There, we use generalised Reynolds operators to produce
Galois-invariant sections to study the stable base locus.
In Section \ref{foli}, we use foliations to tackle the inseparable case.

 \subsection{Acknowledgements}

 The author would like to thank \sloppy \emph{Johan Commelin, Fabio Bernasconi, Andreas Demleitner, Luca Di Cerbo,
 Erwan Rousseau and Roberto Svaldi} for taking their time to read earlier drafts of
 this paper and giving the author useful comments and insights.

 Special thanks to \emph{Calum Spicer} for the insight into tackling the inseparable case; and
 to the author's supervisor \emph{Stefan Kebekus} for making this article even remotely readable!

\section{Preliminaries} \label{prelim}
This section contains the necessary background on the stable and augmented base locus and foliations.

\begin{convention}
Throughout we work with a perfect field $k$.
A \emph{$k$-variety} is an irreducible, separated and integral scheme of finite type over $k$.
We will, in Section \ref{foli}, come across morphisms between $k$-varieties that are not
morphism over $k$. To emphasis this, we will say \emph{$k$-morphism} to mean a morphism
between $k$-varieties that is a morphism of schemes over the base field $k$, a
\emph{morphism} is a morphism of schemes that is not necessarily over $k$.
\end{convention}
\begin{rem}
  The perfect assumption on $k$ is crucial for Section \ref{foli}.
\end{rem}

\subsection{Stable and augmented base locus}

We will follow \cite{bir13} for the definition of augmented base
locus for $\bR$-divisors on a projective scheme. We will then prove a couple of small
but key lemmas that will be used throughout the later sections.

\begin{defn}[{$\bQ$-divisors, $\bR$-divisors and $\bR$-linear equivalence}]
  Let $Y$ be a $k$-variety and denote the group of Cartier divisors on $Y$ as $\Div(Y)$.
  \begin{enumerate}
    \item A $\bQ$-\emph{divisor} on $Y$ is an element in $\Div(Y) \otimes \bQ$.
    \item An $\bR$-\emph{divisor} on $Y$ is an element in $\Div(Y) \otimes \bR$.
    \item Two $\bR$-divisors $\Delta_1$ and $\Delta_2$ are $\bR${\emph{-linearly equivalent}}
    if $\Delta_1 - \Delta_2 = \sum a_i L_i$ where $a_i \in \bR$ and
    $L_i$ are Cartier divisors linearly equivalent to 0. We denote $\bR$-linear
    equivalence as $\Delta_1 \sim_{\bR} \Delta_2$.
  \end{enumerate}
\end{defn}

\begin{defn}[Stable base locus]\label{stab}
 Let $Y$ be a $k$-variety.
 \begin{enumerate}
   \item \label{stab:1} For a Cartier divisor $L$, the stable base locus is defined as
   \[
    \bB(L) := \Bigg( \bigcap_{m \in \bN} \Bs(mL)   \Bigg)_{\red},
   \]
   where $\Bs(mL)$ is the base locus of the linear system $|mL|$.
   \item \label{stab:2} For a $\bQ$-divisor $D$ and a positive integer $b$ such that
   $bD$ is Cartier, the stable locus is defined as
   \[
    \bB^b(D):= \bB(bD).
   \]
 \end{enumerate}
\end{defn}

  \begin{rem}
    In the setting of Definition \ref{stab:2}, the stable base locus for $D$ does
    not depend on the choice of the integer $b$. This follows from the
    observation, for any Cartier divisor $L$ and any positive integer $m$,
    that $\Bs(L) \supseteq \Bs(mL)$. This implies that $\bB(L) \supseteq \bB(mL)$ and
    by Definition \ref{stab}, we find $\bB(L) \subseteq \bB(mL)$. Hence we will simply denote
    the stable base locus as $\bB(D):= \bB^b(D)$.
  \end{rem}

We will use the definition of the augmented base locus for $\bR$-divisors from
\cite[Definition 1.2]{bir13}.

\begin{defn}[Augmented base locus]
  Let $\Delta$ be an $\bR$-divisor on a projective $k$-variety $Y$ and $A$ an ample Cartier divisor on $Y$.
  Choose a representation $\Delta \sim_{\bR} \sum t_i A_i$ where
  $t_i \in \bR $ and $A_i$ are very ample Cartier divisors. Then, the \emph{augmented base locus} is
  \[
    \bB^{A, \sum t_i A_i}_+(\Delta) := \Bigg( \bigcap_{m \in \bN} \bB\bigl (\langle m\Delta \rangle - A \bigr ) \Bigg)_{\red}~,
  \]
  where $\langle m\Delta \rangle = \sum \lfloor mt_i \rfloor A_i$.
\end{defn}

\begin{rem}
  The augmented base locus does not depend on the choice of an ample Cartier divisor $A$,
  or representation $\Delta \sim_{\bR} \sum t_i A_i$ \cite[Lemma 3.1]{bir13}. For this
  reason we drop the superfluous superscripts to simply denote the augmented base locus
  as
  \[
    \bB_+(\Delta) := \bB^{A, \sum t_i A_i }_+(\Delta).
  \]
  By \cite[Lemma 3.1]{bir13} this definition is consistent with
  the definitions found in \cite{Laz04ii, ELMNP6, ELMNP09}.
\end{rem}

\begin{notation}[Pullback of the stable and augmented base locus]
  In the setting of Definition \ref{stab}, let $g: Z \lra Y$ be any surjective finite morphism
  that is not necessarily a $k$-morphism. We abuse notation by denoting $g^{-1}\bB(D)$
  as the reduced scheme associated to the scheme theoretic pullback of $\bB(D)$. We
  do the same with the augmented base locus.
\end{notation}

\begin{rem} The following Lemma \ref{augmentedred} will allow us to reduce statement \ref{mainthm2:2}
  to statement \ref{mainthm2:1} of Theorem \ref{mainthm2}.
\end{rem}

\begin{lem}\label{augmentedred}
  Let $X$ and $Y$ be projective $k$-varieties and $f:X \lra Y$ be a finite morphism\footnotemark{}.
  Suppose that for all Cartier divisors $L$ we have $\bB(f^*L) = f^{-1}\bB(L)$.
  Then for all $\bR$-divisors $\Delta$, we have $\bB_{+}(f^*\Delta) = f^{-1}\bB_{+}(\Delta)$.
\end{lem}

\begin{proof}
  To set up the proof, let $A$ be an ample Cartier divisor on $Y$.
  Let $\Delta$ be an $\bR$-divisor and choose a representation
  $\Delta \sim_{\bR} \sum t_i A_i$ where $t_i \in \bR $ and the $A_i$ are very ample
  Cartier divisors such that $f^*A_i$ are very ample. This is possible, as the pullback
  of an ample divisor under a finite morphism is ample \cite[Proposition 1.2.13]{Laz04i}.
  Observe then that $f^*\Delta \sim_{\bR} \sum t_i f^*A_i$. By this choice of representation,
  we find that for all $m$, we have $f^* \langle m\Delta \rangle  = \langle m f^*\Delta \rangle $.

  Finally we compute that
  \begin{align*}
    f^{-1}\bB_+(\Delta) &= \bigcap_{m \in \bN} f^{-1}\bB\bigl (\langle m\Delta \rangle - A \bigr )  \\
       &= \bigcap_{m \in \bN} \bB\bigl (f^*\langle m\Delta \rangle - f^*A \bigr ) \ \ \ \ \textmd{by hypothesis}\\
                   &= \bigcap_{m \in \bN} \bB\bigl (\langle mf^*\Delta \rangle - f^*A \bigr )
                   = \bB_+(f^*\Delta). \qedhere
  \end{align*}
\end{proof}
The following is a trivial statement. As we will use it several times, we
formulate it as a lemma.

\begin{lem}\label{trivial}
  Let $X$ and $Y$ be normal $k$-varieties and $D$ a $\bQ$-divisor on $Y$.
  For any finite surjective morphism\footnotemark[\value{footnote}] $f:X \lra Y$ we find that $\bB(f^*D) \subseteq f^{-1}\bB(D)$. \qed \qedhere
\end{lem}

\footnotetext{not necessarily a $k$-morphism}

\begin{cor}\label{bigsqueeze}
  Let $X, Y$ and $Z$ be normal $k$-varieties, let $Z \xrightarrow{g} X \xrightarrow{f} Y$
  be a composition of finite surjective morphisms\footnotemark[\value{footnote}], and let $D$ be a $\bQ$-divisor on $Y$.
  If we have $\bB \bigl ( (f \circ g) ^*L \bigr ) = (f \circ g)^{-1}\bB(L)$ then $\bB(f^*L) = f^{-1}\bB(L)$ holds.
\end{cor}
\begin{proof}
  There is a chain of inclusions $\bB \bigl ( (f \circ g)^*D) \subseteq g^{-1}\bB(f^* D)
  \subseteq g^{-1} \bigl (f^{-1}\bB(L) \bigr )$.
  By hypothesis $\bB \bigl ((f \circ g)^*D \bigr ) = (f \circ g)^{-1} \bB(D)$, hence the chain of inclusions are
  indeed equalities. The result follows from applying $g$ to the second inclusion.
\end{proof}

\subsection{Foliation}
We recall the definition of a foliation.

\begin{defn}[Foliation]
  Let $X$ be normal $k$-variety and denote the sheaf of Kähler differentials on $X$ by
  $\Omega^1_X$. A \emph{foliation} on $X$ is a saturated subsheaf of the tangent sheaf
  $\sT_X := \bigl ( \Omega^1_X \bigr )^\vee$ that is closed under the Lie bracket.
\end{defn}

\begin{defn}[{$p$-closed Foliation}]
  Let $\sF$ be a foliation on a normal $k$-variety $X$. Suppose that the characteristic
  of the field $k$ is $p>0$. We say that $\sF$ is \emph{$p$-closed} if
  for all $U \subseteq X$ open and $\forall \partial \in \sF(U)$,
  the composition of $\partial$ with itself $p$-times is in $\sF(U)$.
\end{defn}

\section{The case of separable morphisms}\label{symsec}

  This section is devoted to proving the following result:

  \begin{thm}\label{thm2}
    Let $f: X \lra Y$ be a finite, surjective and separable $k$-morphism of normal projective
    $k$-varieties. Then for every $\bQ$-divisor $D$ on $Y$, we have $\bB(f^*D) = f^{-1}\bB(D)$.
  \end{thm}

  \subsection{Symmetric Sections}

  To prove Theorem \ref{thm2} we will first study the partial case when $f: X \lra Y$ is a finite \emph{Galois}
  morphism, that is when $f$ is isomorphic to a quotient of $X$ by a finite group
  $G$. In this case we make use of the group action $G$ to produce a plethora of
  sections by generalising the classical Reynolds operator.

  Suppose that $L$ is a Cartier divisor on $Y$ and $\sigma \in H^0 \bigl (X, \sO_X(f^*L) \bigr )$, we
  can apply the classical Reynolds operator on $\sigma$ to obtain a section,
  which we denote by, $R_1(\sigma) \in H^0 \bigl (Y, \sO_Y(L) \bigr )$. Even if $\sigma(x) \neq 0$
  for some $x \in X$, there is no guarantee that $R_1\bigl (\sigma \bigr)(f(x)) \neq 0$.
  Hence it is hard to control the stable base locus with the Reynolds operator. The following
  Construction \ref{symmetricsections} remedies this issue.

  \begin{construction}\label{symmetricsections}
     Let $X$ be a projective $k$-variety and $G < \Aut_k(X)$ be a non trivial finite group.
     Denote by $f: X \lra Y$ the quotient of $X$ by $G$. Take a Cartier divisor $L$ on $Y$.	Write $G$ as
	   \[
		   G = \{g_1, g_2, \ldots, g_N \},
	   \]
     such that $g_1$ is the identity element.
     For each $1 \leq i \leq N$ there is a linear map
  \begin{align*}
		R_i : H^0 \bigl (X, \sO_X(f^*L) \bigr ) & \lra H^0 \bigl ( X, \sO_X(i \cdot f^*L) \bigr ) = H^0 \bigl ( X, \Sym^i\sO_X( f^*L) \bigr ) \\
						\sigma & \longmapsto
						 \sum_{1 \leq j_1 < j_2 < \cdots < j_i \leq N} g_{j_1}^*\sigma \otimes g_{j_2}^*\sigma \otimes
						 \cdots \otimes g_{j_i}^*\sigma.
	\end{align*}
	It should be clear that $R_i(\sigma)$ are $G$-invariant by construction. It will
	also be useful to define the following linear maps for $i = 1, \ldots, N-1$:

	\begin{align*}
		\tau_i : H^0 \bigl (X, \sO_X(f^*L) \bigr ) & \lra H^0 \bigl ( X, \sO_X(i \cdot f^*L) \bigr ) \\
						\sigma & \longmapsto
						 \sum_{1 < j_1 < j_2 < \cdots < j_i \leq N} g_{j_1}^*\sigma \otimes g_{j_2}^*\sigma \otimes
						 \cdots \otimes g_{j_i}^*\sigma
	\end{align*}
\end{construction}

\begin{rem}
  In Construction \ref{symmetricsections}, the definitions of $R_i(\sigma)$ and $\tau_i(\sigma)$
  do not depend on the choice of ordering of $G$.
\end{rem}

\begin{lem}\label{relations} In the setting of Construction \ref{symmetricsections},
   given a section $\sigma \in H^0 \bigl ( X, \sO_X(f^*L) \bigr )$ there are relations:
   \[
   R_i(\sigma) =
	\begin{cases}
		 \sigma + \tau_1(\sigma) & \textmd{ for $i = 1$}\\
		 \sigma \otimes \tau_{i-1}(\sigma) +  \tau_{i}(\sigma) & \textmd{ for $1 < i < N$}\\
		 \sigma \otimes \tau_{N-1}(\sigma) & \textmd{ for $i = N$}.
	\end{cases}
  \]
\end{lem}
\begin{proof}
	The cases $i = 1$ and $i = N$ are clear. For $1<i<N$, we can split $R_i(\sigma)$ as follows
	\begin{align*}
		R_i(\sigma) &= \sum_{1 \leq j_1 < j_2 < \cdots< j_i \leq N} g_{j_1}^*\sigma \otimes g_{j_2}^*\sigma \otimes
		\cdots \otimes g_{j_i}^*\sigma \\
		  &= \sum_{1 = j_1 < j_2 < \cdots< j_i \leq N} g_{j_1}^*\sigma \otimes g_{j_2}^*\sigma \otimes
			\cdots \otimes g_{j_i}^*\sigma \\
			& \ \ \ \ \ \ \ \ \ \ \ \ \ \ \ \ \ \ + \sum_{1 < j_1 < \cdots < j_i \leq N} g_{j_1}^*\sigma \otimes g_{j_2}^*\sigma \otimes
			\cdots \otimes g_{j_i}^*\sigma \\
		 &= \sigma \otimes \tau_{i-1}(\sigma) +  \tau_{i}(\sigma). \qedhere
	\end{align*}
\end{proof}

\begin{lem}\label{pointnonvanishing}
	In the setting of Construction \ref{symmetricsections},
  let $\sigma \in  H^0 \bigl ( X, \sO_X(f^*L)\bigr ) $ and $x \in X$ be a scheme theoretic point such that
  $\sigma(x) \neq 0$. Then, there exists an $1 \leq i \leq N$ such that $R_i(\sigma)(x) \neq 0$.
\end{lem}

\begin{proof}
	Suppose that $R_i(\sigma)(x) = 0$ for all $i = 1, \ldots, N-1$.
	Then by the relations	of Lemma \ref{relations}, one finds that
	\[
		R_N(\sigma)(x) =   (-1)^{(N-1)} \sigma^{\otimes N}(x).
	\]
	This implies that $R_N(\sigma)(x)$ is non-zero as it is assumed that $\sigma(x) \neq 0$.
	Hence $R_i(\sigma)(x)$ cannot simultaneously vanish for all $1 \leq i \leq N$.
\end{proof}

\subsection{{Proof of Theorem \ref{thm2}}}

  Lemma \ref{bigsqueeze} allows us to pass to the Galois closure and, without loss of generality,
  assume that $f:X \lra Y$ is the quotient of $X$ by a finite group $G$.

  The inclusions $\bB(f^*L) \subseteq f^{-1}\bB(L)$ follows from Lemma \ref{trivial}.

  Conversely, take $x \notin \bB(f^*D)$. This means that there exists an $m \in \bN$
  and a section $\sigma \in H^0 \bigl (X, \sO_X(m \cdot f^*D) \bigr ) $ such that $\sigma(x) \neq 0$.
  To show that $x \notin f^{-1}\bB(D)$, we need to produce a section in
  $H^0 \bigl (Y, \sO_Y( n \cdot D) \bigr )$, for some $n>0$, that does not vanish at $f(x)$. By
  Lemma \ref{pointnonvanishing}, we find that there exists an $i$, such that
  $R_i(\sigma) \in f^*H^0 \bigl (Y, \sO_Y(mi \cdot D) \bigr )$ satisfying $R_i(\sigma)(x) \neq 0$. \qed

  \section{The Case of inseparable morphisms} \label{foli}

  The goal of this section is to remove the separability hypothesis of Theorem~\ref{thm2},
  and hence proving Theorem \ref{mainthm2}.
  In characteristic $0$, every finite morphism is automatically separable. In positive
  characteristic, a finite field extension splits into a separable extension and
  a purely inseparable extension. Hence it enough to study purely inseparable
  morphisms. In this case, the idea is to use Ekedahl's theory of foliations and
  purely inseparable morphisms to factorise a purely inseparable $k$-morphism
  into purely inseparable morphisms of height one \cite{Ek87}. We then study the
  pullback of the stable and augmented base locus via a purely inseparable morphism
  of height one.

  \begin{setting}\label{poscar}
    Assume that the perfect field $k$ has characteristic $p > 0$.
  \end{setting}

  \subsection{Ekedahl's theory of foliations and purely inseparable morphisms}

  The original paper \cite{Ek87} covers the setting of smooth $k$-varieties. However,
  it is crucial to work in the setting of normal $k$-varieties. The excellent \cite{PatWa18}
  gives a nice exposition of this more general setting in \cite[Section 2.4]{PatWa18},
  which we follow closely.

  \begin{defn}[Absolute Frobenius] In Setting \ref{poscar}, let $X$ be a
    normal $k$-variety. The \emph{absolute Frobenius} morphism
    $F_X: X \lra X$ is the endomorphism of schemes induced by the identity on the underlying
    topological space and the morphism between the sheaves of rings, $F^\#_X: \sO_X \lra \sO_X$,
    is given by sending $s \mapsto s^p$.
  \end{defn}

  \begin{rem}
    In general, the absolute Frobenius morphism is not a $k$-morphism. The pushforward
    $\bigl (F_X \bigr )_*\sO_X$ is isomorphic to $\sO_X$ as sheaves. However the $\sO_X$-module
    structure on $\bigl (F_X \bigr )_* \sO_X$ is given by the rule
    \begin{align*}
        \sO_X \times \bigl (F_X \bigr )_*\sO_X &\lra \bigl (F_X \bigr )_* \sO_X \\
           (r,s) \quad \quad \ \  &\longmapsto \ \ r^p \cdot s.
    \end{align*}
    We will label $\bigl (F_X \bigr )_*\sO_X$ with this $\sO_X$-module structure
    by $\sO_{F_X}$.
  \end{rem}

  \begin{defn}[Purely inseparable morphism]
    In Setting \ref{poscar}, a finite surjective morphism\footnotemark{} $g:X \lra Z$ between normal
    $k$-varieties is called \emph{purely inseparable} if the field extension
    $k(Z) \subseteq k(X)$ is a purely inseparable extension.
  \end{defn}

  \begin{defn}[Purely inseparable morphism of height one]In Setting \ref{poscar}, a finite surjective
    morphism\footnotemark[\value{footnote}] $g: X \lra Z$ between normal $k$-varieties is called \emph{purely inseparable of height one} if
    it is purely inseparable and there exists a morphism $h: Z \lra X$ such that $h \circ g = F_X$, where
    $F_X$ is the absolute Frobenius of X.
  \end{defn}

  \footnotetext{not necessarily a $k$-morphisms}

  \subsubsection{Foliations and finite morphisms}

  Here we highlight how to obtain a foliation from a finite morphism and vice-versa.

  \begin{construction}[Quotient by a foliation]\label{foltomor} In Setting \ref{poscar},
    let $X$ be a normal $k$-variety and $\sF$ a foliation on $X$. The functions of
    $X$ that are annihilated by $\sF$ define a sheaf $\sA_{\sF}$. More precisely,
    the assignment sending each open set $U \subseteq X$ to
    \[
    \sA_\sF (U) := \{s \in \sO_X(U) \mid \textmd{$\forall V \subseteq U$ open, $\forall \partial \in \sF(V) : \partial s = 0$} \},
    \]
    defines $\sA_{\sF}$.
    The $\sO_X$-module structure of $\sA_{\sF}$ is defined by the multiplication rule
    \begin{align*}
        \sO_X \times \sA_\sF &\lra  \ \sA_\sF \\
           (r,s) \quad  &\longmapsto r^p \cdot s
    \end{align*}
    There is an inclusion of $\sO_X$-modules $\sA_\sF \hookrightarrow \sO_{F_X}$.
    Taking the relative spectrum of this inclusion we obtain
    \[
      \begin{tikzcd}
        X \arrow[bend left = 15]{rr}{F_X} \ar[d] \ar[r] & \Spec_X \sA_\sF \ar[d] \ar[r] &  X \ar[d] \\
        \Spec k \ar[r, equals] & \Spec k \ar[r, "F_{\Spec k}"] & \Spec k.
      \end{tikzcd}
    \]
    The $k$-morphism $X \lra \Spec_X \sA_{\sF}$ is called the \emph{quotient of $X$ by $\sF$}.
  \end{construction}

  \begin{construction}[Foliation associated to a finite morphism] \label{finfoliation} In Setting \ref{poscar}, let
    $f:X \lra Y$ be a finite surjective $k$-morphism between normal $k$-varieties. The
    sections of $\sT_X$ which annihilate $f^{-1}\sO_Y \subseteq \sO_X$ define a subsheaf
    of $\sT_X$. More precisely, the assignment sending each open set $U \subseteq X$ to
    \[
      \sF_f(U) := \{ \partial \in \sT_X (U) \mid \textmd{$\forall V \subseteq U$ open, $ \forall s \in f^{-1}\sO_Y(V) : \partial s = 0 $} \}
    \]
    defines a subsheaf of $\sO_X$-modules of $\sT_X$.
    \end{construction}

    \begin{lem}
      In Construction \ref{finfoliation}, the sheaf $\sF_f$ is a $p$-closed foliation.
    \end{lem}

    \begin{proof}
      To show that $\sF_f$ is closed under the Lie bracket, fix open sets
      $V \subseteq U \subseteq X$ and a function $s \in f^{-1}\sO_Y(V)$. For all $\partial, \partial' \in \sF_f(U)$,
      compute the Lie bracket to find
      \[
      [\partial, \partial'](s) = \partial \bigl ( \partial'(s) \bigr ) - \partial' \bigl (\partial(s) \bigr ) = 0.
      \]
      A similar computation shows that $\sF_f$ is $p$-closed.

      Now we show that $\sF_f$ is a saturated in $\sT_X$. If $\sF_f = \sT_X$ or $\sF_f = 0$, then
      there is nothing to prove, so let us assume $\sF_f$ is a proper subsheaf. Fix an open set
      $U \subseteq X$. Suppose for a contradiction that ${\tau}$ is a \emph{non-zero} torsion
      section in $\factor{\sT_X}{\sF_f}(U)$. By shrinking $U$ to a smaller open set, if necessary,
      we can choose a lift  $\partial \in \sT_X(U)$ of $\tau$.
      Since $\tau$ is non zero, this implies that $\partial \notin \sF_f(U)$. Since
      $\tau$ is a torsion section, there exists a non zero function $r \in \sO_X(U)$
      such that $r \cdot \tau = 0$. This means that $r \cdot \partial \in \sF_f(U)$.

      To arrive at the contradiction, pick a section $t \in f^{-1}\sO_Y(V)$ for some
      $V \subseteq U$ open such that its evaluation $\partial(t) \neq 0$ however then $\bigl (r \cdot \partial \bigr )(t) = 0$.
      Since $X$ is integral, as it is normal, this leads to contradiction as it suggests
      that $r$ is a zero divisor.
    \end{proof}

    \subsubsection{Ekedahl's Theorem}

    Restricting the class of morphisms in Construction \ref{finfoliation}, we obtain
    a correspondence between $p$-closed foliations and purely inseparable $k$-morphisms of
    height one.

   \begin{rem}
    It is a necessary condition that $k$ is a perfect field in Theorem \ref{ekedahl} below.
   \end{rem}

  \begin{thm}[{\cite[Proposition 2.9]{PatWa18}}] \label{ekedahl} In Setting \ref{poscar}, let $X$ be a normal $k$-variety.
    There is a 1-1 correspondence between
    \begin{enumerate}
      \item $p$-closed foliations $\sF \subset \sT_X$, and
      \item purely inseparable $k$-morphisms $X \lra Z$ of height one.
    \end{enumerate}
    The correspondence is given by:
    \begin{itemize}
      \item sending the foliation $\sF$ to the quotient of $X$ by $\sF$, and
      \item sending the morphism $X \xrightarrow{g} Z$ to the $p$-closed foliation $\sF_g$.
    \end{itemize}
    Moreover the correspondence satisfies $[k(Z):k(X)] = p^{\rank \sF}$. \qed \qedhere
  \end{thm}

  \begin{rem}
    Theorem \ref{ekedahl} asserts that the quotient by a $p$-closed foliation is in particular
    a normal variety, and the quotient morphism is finite and surjective.
  \end{rem}

  \subsection{Factorising purely inseparable morphisms}

  Using Theorem \ref{ekedahl}, we will factorise a purely inseparable $k$-morphism
  into height one pieces.

  \begin{prop}\label{factorisation}
    Let $f:X \lra Y$ be a purely inseparable $k$-morphism between normal
    $k$-varieties. Then there exists a positive integer $l$ and normal $k$-varieties $Z_1, ..., Z_l$
    that factorise $f$
      \[
        X = Z_1 \lra Z_2 \lra \cdots \lra Z_l = Y.
      \]
    The morphisms $Z_{\bullet} \lra Z_{\bullet + 1}$ are
    purely inseparable $k$-morphisms of height one.
  \end{prop}

  To prove Proposition \ref{factorisation}, we will need the following technical
  lemma.

  \begin{lem}\label{nonzero} In Construction \ref{finfoliation}, the foliation $\sF_f$ is
    zero if and only if $f$ is a separable morphism.
  \end{lem}

  \begin{rem}
    The consequence of Lemma \ref{nonzero} is that, the quotient by $\sF_f$ is the
    identity morphism if and only if $f$ is a separable morphism.
  \end{rem}

  We will prove Lemma \ref{nonzero} after we give a proof of
  Proposition \ref{factorisation}.

  \begin{proof}[Proof of Proposition \ref{factorisation}]
    If the morphism $f:X \lra Y$ is the identity, we are done. Supposing otherwise,
    we will show that then it is possible to factorise $f: X \lra Y$ into
    \begin{equation}\label{purefactor}
      X \xrightarrow{g} Z_{2} \xrightarrow{f_2} Y
    \end{equation}
    where $g$ is a purely inseparable $k$-morphism of height one and $f_2$ is a
    purely inseparable $k$-morphism.
    The morphism $g$ is simply the quotient of the foliation associated to the
    finite morphism $f: X \lra Y$, which by Theorem \ref{ekedahl} is purely
    inseparable of height one and by Lemma \ref{nonzero} is not the identity
    morphism.

    Using the fact that purely inseparable morphisms are universal homeomorphisms
    \cite[Exercise 12.32]{GW10}, we find that $f$ and $g$ are universal
    homeomorphisms and that the underlying topological spaces of $Y$, $X$ and $Z_{2}$
    are all homeomorphic. Hence the morphism $f_{2}: |Z_2| \lra |Y|$, between
    the underlying topological spaces is given  by $g^{-1} \circ f$.

    For the morphism of sheaves of rings, observe that there is an inclusion of sheaves
    $f^{-1} \sO_Y \subseteq g^{-1} \sO_{Z_{2}} \subseteq \sO_{X}$.
    The sheaf $g^{-1} \sO_{Z_{2}}$ is the sheaf of functions annihilated by the
    foliation $\sF_{f}$, which contains $f^{-1}\sO_Y$, as in Construction \ref{foltomor}.
    For an open set $V \subseteq X$ we obtain a ring homomorphism
    \begin{equation}\label{sheafmorph}
      \sO_Y\bigl ( f(V)\bigr ) \hookrightarrow \sO_{Z_{2}} \bigl ( g(V) \bigr).
    \end{equation}
    For any open set $U \subseteq Z_{2}$, by setting $V = g^{-1}(U)$ in Equation
    \eqref{sheafmorph}, we obtain a well defined morphism of sheaves
    $f_{2}^{\#}: f^{-1}_{2}\sO_Y \lra \sO_{Z_{2}}$. Hence we get a morphism of
    schemes $f_{2}: Z_{2} \lra Y$ factorising $f$ as found in Equation
    \eqref{purefactor}.

    We proceed by applying the same argument to $f_2: Z_2 \lra Y$. Since $f: X \lra Y$
    a finite morphism, this process of factorising must stop after a finite number
    of steps. This proves the proposition.
  \end{proof}

  In preparation to prove Lemma \ref{nonzero}, it will be useful to give an alternative
  description of $\sF_f$ in terms of differentials, to give us a concrete connection
  between $\sF_f$ and the separability of $f$.

  \begin{construction}\label{Ann} In the Setting of Construction \ref{finfoliation},
    Let $\sM$ be the image of the map $f^* \Omega^1_Y \lra \Omega^1_X$.
    Denote by $\Ann{\sM}$ the sections of $\sT_X$ that are annihilated by $\sM$.
    This is described in the following exact sequence:
    \begin{equation}\label{Annseq}
    0 \lra \Ann{\sM} \lra \sT_X \lra \bigl (\sM \bigr )^\vee .
    \end{equation}
    Moreover, since the quotient of $f^*\Omega^1_Y$ by $\sM$ is $\Omega^1_{X/Y}$, we find
    that $\Ann{\sM} = \bigl (\Omega^1_{X/Y} \bigr )^{\vee}$.
  \end{construction}

  \begin{lem}\label{conection} The sheaves $\sF_f$ and $\Ann \sM$ constructed in Construction
    \ref{finfoliation} and Construction \ref{Ann} are isomorphic.
  \end{lem}

  \begin{proof}
    Let $U \subseteq X$ be an open set and fix a section $s \in f^{-1} \sO_{Y}(U)$
    and fix a derivation $\partial \in \Ann \sM(U)$.
    Observe that $ds \in \sM (U)$. We can compute that $\partial (ds) = \partial (s) = 0$,
    which shows that $\partial \in \sF_f(U)$.

    Conversely, fix a derivation $\partial \in \sF_f(U)$ for any open set $U \subseteq X$.
    We need to show that for any $\omega \in \sM (U)$, that $\partial \omega = 0$. Since
    $\sM$ is locally generated by sections of the form $f^* dy$, where $y$ is a local section of $\sO_Y$;
    it is enough to check that $\partial (f^* dy) = 0$. This follows from the
    computation $\partial (f^* dy) = f^* (\partial dy) = 0$, since $f^*y$ is
    a local section of $f^{-1}\sO_Y$.
  \end{proof}

  \begin{proof}[Proof of Lemma \ref{nonzero}]
    By Lemma \ref{conection} and Construction \ref{Ann}, we get that
    $\sF_f = \Ann \sM = \bigl (\Omega^1_{X/Y} \bigr )^{\vee}$. This allows us to work
    with relative differentials.

    Suppose that $f:X \lra Y$ is separable. Since $\Omega^1_{X/Y}$ is supported
    on the ramification locus of $f$, it is torsion. Therefore the dual
    $\bigl (\Omega^1_{X/Y} \bigr )^{\vee} = \sF_f = 0$.

    Conversely, suppose that $f$ is not separable. Denote the generic
    point of $X$ by $\eta$. We compute that
    \begin{align*}
      \bigl (\Ann{\sM} \bigr )_{\eta} = \Bigl (\bigl ( \Omega^1_{X/Y} \bigr )^{\vee} \Bigr )_{\eta}
                       &= \Bigl ( \bigl (\Omega^1_{X/Y}\bigr)_{\eta} \Bigr )^{\vee} &\textmd{\cite[III Propostion 6.8]{Hart}} \\
                       &= \bigr( \Omega^1_{K(X)/K(Y)})^{\vee} \neq 0  &\textmd{\cite[II Theorem 8.6A]{Hart}}.
    \end{align*}
    Hence $\sF_f \neq 0$.
  \end{proof}

  \subsection{Stable base locus under purely inseparable morphisms}
  In this subsection, we make use of the Frobenius morphism to understand the
  pullback of the stable base locus via a purely inseparable height one morphism.

  \begin{lem}\label{frobase}
    In Setting \ref{poscar}, let $F_X: X \lra X$ be the absolute Frobenius
    morphism and let $D$ be a $\bQ$-divisor on $X$. Then $\bB(F_X^*D) = F_X^{-1}\bB(D)$.
  \end{lem}

  \begin{proof}
    The left hand side of the expression is simply $\bB(pD)$ since $F_X^*D = pD$.
    The right hand side is $\bB(D)$ as the stable base locus is a closed subset of
    $X$ and the absolute Frobenius morphism leaves the points of $X$ fixed. The
    expression immediately follows from the observation that $\bB(pD) = \bB(D)$.
  \end{proof}

  \begin{lem}\label{insapbase}
    In Setting \ref{poscar}, let $g:X \lra Z$ be a purely inseparable
    height one morphism\footnote{not necessarily a $k$-morphism}
    and $D$ a $\bQ$-divisor on $Z$. Then $\bB(g^*D) = g^{-1}\bB(D)$.
  \end{lem}

  \begin{proof}

    Since $g$ is purely inseparable of height one, there exists a morphism
    $ Z \xrightarrow{i} X$ such that $g \circ i = F_X$. We claim that
    $i \circ g = F_Z$, the Frobenius endomorphism of $Z$. Then by Lemma \ref{frobase}
    and Lemma \ref{bigsqueeze} the result follows.

    To prove the claim, observe that $F_X \circ g = g \circ F_Z$ by
    \cite[\href{https://stacks.math.columbia.edu/tag/0CC7}{Tag 0CC7}]{stacks-project}.
    Since $F_X = g \circ i$, if we can show that $g$ is an epimorphism, or in other
    words left cancellative, then the claim follows. This follows from the observation
    that $g$ is surjective and the map of sheaves of rings $g^{\#}: \sO_Z \lra \sO_X$
    is injective.
  \end{proof}

  \subsection{Proof of Theorem \ref{mainthm2}}

  We start with a series of reduction steps. It is enough to prove Statement~\ref{mainthm2:1}
  due to Lemma \ref{augmentedred}.

  In characteristic $0$, the morphism $f:X \lra Y$ is necessarily separable, this
  case is covered by Theorem~\ref{thm2}. We can then assume we are in Setting \ref{poscar},
  namely in positive characteristic.

  Since any finite field extension can be split into a separable extension followed
  by a purely inseparable extension \cite[\href{https://stacks.math.columbia.edu/tag/030K}{Tag 030K}]{stacks-project},
  we can assume without loss of generality that $f: X \lra Y$ is a purely inseparable extension.

  By Proposition \ref{factorisation}, we obtain a factorisation of $f$
  \[
    X \lra Z_1 \lra \cdots \lra Z_l \lra Y.
  \]
  Since all the morphisms in the factorisation are purely inseparable of
  height one, by Lemma \ref{insapbase} we get $\bB(f^*D) = f^{-1}\bB(D)$. \hfill   \qed


\newcommand{\etalchar}[1]{$^{#1}$}
\providecommand{\bysame}{\leavevmode\hbox to3em{\hrulefill}\thinspace}
\providecommand{\MR}{\relax\ifhmode\unskip\space\fi MR }
\providecommand{\MRhref}[2]{%
  \href{http://www.ams.org/mathscinet-getitem?mr=#1}{#2}
}
\providecommand{\href}[2]{#2}

\end{document}